\documentclass[10pt]{amsart}

\usepackage{dsfont}
\usepackage{amssymb}
\usepackage{mathtools}
\usepackage{tikz}
\usepackage{mathrsfs}

\title{Absolute Continuity and Large-Scale Geometry of Polish Groups}
\author{Jake Herndon}

\newcommand{\R}{\mathbb{R}}
\newcommand{\Z}{\mathbb{Z}}
\newcommand{\loc}{\mathrm{loc}}

\DeclareMathOperator{\Bij}{Bij}
\DeclareMathOperator{\Homeo}{Homeo}
\DeclareMathOperator{\Diff}{Diff}
\DeclareMathOperator{\Isom}{Isom}
\DeclareMathOperator{\AC}{AC}

\theoremstyle{plain}
\newtheorem{thm}{Theorem}
\newtheorem{lemma}[thm]{Lemma}
\newtheorem{prop}[thm]{Proposition}
\newtheorem{coro}[thm]{Corollary}

\theoremstyle{definition}
\newtheorem{example}[thm]{Example}
\newtheorem{question}[thm]{Question}
\newtheorem{notation}[thm]{Notation}

\begin{document}

\begin{abstract}
We apply the theory of large-scale geometry of Polish groups to groups of absolutely continuous homeomorphisms. Let $M$ be either the compact interval or circle.  We prove that the Polish group $\AC_+(M)$ of orientation-preserving homeomorphisms $f:M\to M$ such that $f$ and $f^{-1}$ are absolutely continuous has a trivial quasi-isometry type.  We also prove that the Polish group $\AC_\Z^\loc(\R)$ of homeomorphisms $f:\R\to\R$ such that $f$ commutes with integer translations and both $f$ and $f^{-1}$ are locally absolutely continuous is quasi-isometric to the group of integers.  To study $\AC_+\left(\mathbb S^1\right)$ and $\AC_\Z^\loc(\R)$ we use the observation that these groups are Zappa-Sz\'ep products.
\end{abstract}

\maketitle

\section{Introduction}\label{sect: introduction}

Throughout this paper $I,\mathbb S^1,\R,$ and $\Z$ denote the compact interval, the circle, the real line, and the integers, respectively.  For $M=I$ and $M=\mathbb S^1$ we let $\AC_+(M)$ denote the group of orientation-preserving homeomorphisms $f:M\to M$ such that both $f$ and $f^{-1}$ are absolutely continuous, and we let $\AC_\Z^{\loc}(\R)$ denote the group of homeomorphisms $f:\R\to\R$ such that $f$ commutes with integer translations and both $f$ and $f^{-1}$ are locally absolutely continuous.  A {\em Polish group} is, as usual, a topological group whose underlying topology is Polish, i.e. separable and completely metrizable, and the underlying topology is called a {\em Polish group topology}.  In \cite{Solecki} Solecki defines a Polish group topology on $\AC_+(I)$, and here we define a Polish group topology on $\AC_+\left(\mathbb S^1\right)$ and another Polish group topology on $\AC_\Z^{\loc}(\R)$.  Henceforth we also use $\AC_+(I),\AC_+\left(\mathbb S^1\right)$, and $\AC_\Z^{\loc}(\R)$ to denote these Polish groups.

The goal of this paper is to study $\AC_+(I),\AC_+\left(\mathbb S^1\right),$ and $\AC_\Z^{\loc}(\R)$ from the perspective of large-scale geometry.  There is a general theory of large-scale geometry of Polish groups which is due to Rosendal and detailed in \cite{RoseCoarse}.  The main results of this paper are summarized by the following theorem.

\begin{thm}\label{thm: Main Theorem}
Each of the Polish groups $\AC_+(I),\AC_+\left(\mathbb S^1\right)$, and $\AC_\Z^{\loc}(\R)$ has a well-defined quasi-isometry type.  For both $\AC_+(I)$ and $\AC_+\left(\mathbb S^1\right)$ the quasi-isometry type is trivial and for $\AC_\Z^{\loc}(\R)$ the quasi-isometry type is that of $\Z$.
\end{thm}

The theory of large-scale geometry of Polish groups generalizes the classical theory of large-scale geometry of finitely generated groups by viewing every finitely generated group as a discrete topological group, and therefore, it is helpful to first review the situation for finitely generated groups before making Theorem \ref{thm: Main Theorem} precise.  Recall, for a group $G$ generated by a subset $S$ there is the associated left-invariant {\em word metric} $\rho_S$ defined by
\[\rho_S(x,y)=\min\left\{k\geq 0\vert x^{-1}y\in\left(S\cup S^{-1}\right)^k\right\}\]
for all $x,y\in G$.  If $G$ is finitely generated then the word metrics obtained from finite generating sets are all mutually quasi-isometric, and so we say the quasi-isometry type of $G$ is the quasi-isometry equivalence class of the metric space $(G,\rho_S)$ for any finite generating set $S$.  A running theme in the subject of geometric group theory is the relation of algebraic properties of a finitely generated group with large-scale metric properties of its quasi-isometry type.

A crucial part of the general theory for Polish groups is to identity a collection of distinguished generating sets that, when present, give mutually quasi-isometric word metrics.  Following \cite{RoseCoarse}, a subset $S$ of a Polish group $G$ is {\em coarsely bounded in $G$} if $S$ is bounded in every metric on $G$ which is topologically compatible and left-invariant.  For a Polish group $G$ which is generated by a coarsely bounded subset the {\em quasi-isometry type of $G$} is the quasi-isometry equivalence class of the metric space $(G,\rho_S)$ for any coarsely bounded and closed generating set $S$.  Indeed, by the Baire Category Theorem, any word metrics obtained from generating sets which are coarsely bounded and closed are mutually quasi-isometric.  We note that the collection of all coarsely bounded subsets of a Polish group is closed under taking topological closures, and so any Polish group which is generated by a coarsely bounded subset is also generated by one which is closed.  We also note that the property of a subset being coarsely bounded is not a homeomorphism invariant as the definition quantifies over metrics on the ambient group.  However, if it does not create ambiguity we omit reference to the ambient group.  A Polish group which is a coarsely bounded subset of itself is a {\em coarsely bounded group}.  For Polish groups $G$ and $H$, a function $G\to H$ is a {\em quasi-isometry of Polish groups} if $G$ and $H$ are generated by subsets which are coarsely bounded and $(G,d_G)\to(H,d_H)$ is a quasi-isometry whenever $d_G$ and $d_H$ represent the quasi-isometry types of $G$ and $H$, respectively.  In this case $G$ and $H$ are {\em quasi-isometric Polish groups}.  The coarsely bounded Polish groups are those which are quasi-isometric to the trivial group.

In the context of the above general theory we prove Theorem \ref{thm: Main Theorem} by showing $\AC_+(I)$ and $\AC_+\left(\mathbb S^1\right)$ are coarsely bounded groups and by showing the identification of $\Z$ with the subgroup of $\AC_\Z^\loc(\R)$ consisting of integer translations defines a quasi-isometry of Polish groups $\Z\to\AC_\Z^\loc(\R)$.  Our exposition repeatedly makes use of the observation that for both the case of $G=\AC_+\left(\mathbb S^1\right)$ and $G=\AC_\Z^{\loc}(\R)$ the group $G$ is a Zappa-Sz\'ep product of subgroups $H$ and $K$ with $K$ isomorphic to $\AC_+(I)$.  For this reason Section \ref{sect: ZSP} is a standalone section which contains information relevant to Zappa-Sz\'ep products.  In Section \ref{sect: prelims} we recall the definition of absolute continuity and its local counterpart and we give a compatible, right-invariant metric on each of the Polish groups $\AC_+(I)$, $\AC_+\left(\mathbb S^1\right)$, and $\AC_\Z^\loc(\R)$.  In Section \ref{sect: absolute continuity} we prove Solecki's topology on $\AC_+(I)$ makes it a coarsely bounded group, and in Section \ref{sect: AC(S) and AC_Z(R)} we verify that the topologies defined on $\AC_+\left(\mathbb S^1\right)$ and $\AC_\Z^\loc(\R)$ are, in fact, Polish group topologies and then we complete the proof of Theorem \ref{thm: Main Theorem}.

Let us mention some other groups of homeomorphisms where the large-scale perspective is relevant.  For a compact manifold $M$ the group $\Homeo(M)$ of homeomorphisms of $M$ is a Polish group with the compact-open topology.  Let $M$ be a compact manifold and let $\Homeo_0(M)$ denote the identity component in $\Homeo(M)$.  In \cite{RoseCoarse} Mann and Rosendal study the large-scale geometry of $\Homeo_0(M)$.  The main results there include proving $\Homeo_0(M)$ is generated by a coarsely bounded subset and so has a well-defined quasi-isometry type.  They show if $M$ is the $n$--sphere then $\Homeo_0(M)$ is a coarsely bounded group and if $\dim (M)>1$ and $\pi_1(M)$ is infinite then the quasi-isometry type of $\Homeo_0(M)$ is non-trivial.  Let $k>0$ be an integer or possibly $\infty$ and let $M=I$ or $M=\mathbb S^1$.  The group $\Diff_+^k(M)$ of orientation-preserving $C^k$--diffeomorphisms of $M$ is a Polish group with its $C^k$--topology. In \cite{Cohen} Cohen proves $\Diff_+^k(M)$ is generated by a coarsely bounded subset if and only if $k<\infty$, and in this case the quasi-isometry type of $\Diff_+^k(M)$ is non-trivial.

\section{Preliminaries}\label{sect: prelims}

We identify $I$ with the unit interval $[0,1]$ and $\mathbb S^1$ with the group of complex numbers with unit norm, and we let $\pi:\R\to\mathbb S^1$ be the covering map $x\mapsto e^{2\pi ix}$.  We use $d_{\mathbb S^1}$ to denote the metric on $\mathbb S^1$ which is defined by taking the minimum distance between $\pi^{-1}(x)$ and $\pi^{-1}(y)$ in $\R$ for all $x,y\in\mathbb S^1$

A homeomorphism $f:\R\to\R$ {\em commutes with integer translations} if
\[f(x+n)=f(x)+n\]
for all $x\in\R$ and all $n\in\Z$.  The group of all homeomorphisms of $\R$ that commute with integer translations is denoted $\Homeo_\Z(\R)$.  For $G=\Homeo_+(I)$, $G=\Homeo_+\left(\mathbb S^1\right)$, and $G=\Homeo_\Z(\R)$ we use $d_\infty$ to denote a right-invariant metric on $G$ which induces the compact-open topology.  Hence $G$ is a Polish group with the topology induced by $d_\infty$.  For concreteness we set
\[d_\infty(f,g)=\sup_{x\in I}|f(x)-g(x)|\]
for all $f,g\in\Homeo_+(I)$ and for all $f,g\in\Homeo_\Z(\R)$, and we set
\[d_\infty(f,g)=\sup_{x\in\mathbb S^1}d_{\mathbb S^1}(f(x),g(x))\]
for all $f,g\in\Homeo_+\left(\mathbb S^1\right)$.

A function $f:J\to\R$ whose domain is an interval is {\em absolutely continuous} if for every $\epsilon>0$ there is a $\delta>0$ such that for every finite sequence $(a_1,b_1),\ldots,(a_n,b_n)$ of disjoint subintervals of $J$, if $\sum_{i=1}^n(b_i-a_i)<\delta$ then $\sum_{i=1}^n|f(b_i)-f(a_i)|<\epsilon$.  A function $f:J\to\R$ whose domain as an interval is {\em locally absolutely continuous} if the restriction of $f$ to every compact subinterval of $J$ is absolutely continuous.  For every homeomorphism $f:\mathbb S^1\to\mathbb S^1$ there is a unique homeomorphism $\tilde f:\R\to\R$ with $\tilde f(0)\in[0,1)$ and $\pi\circ\tilde f=f\circ\pi$.  A homeomorphism $f:\mathbb S^1\to\mathbb S^1$ is {\em absolutely continuous} if $\tilde f$ is locally absolutely continuous.  

Suppose $J$ is a compact interval and $f:J\to\R$ is continuous and nondecreasing.  The Fundamental Theorem of Lebesgue Integration states that $f$ is absolutely continuous if and only if $f$ has a derivative $f'$ almost everywhere on $J$ with respect to Lebesgue measure, $f'\in L^1(J)$, and for all $x\in J$
\[f(x)=f(a)+\int_a^x f'(t)\ dt.\]
See, for instance, \cite[Theorem 7.18]{Rudin}.

We set
\[d_*(f,g)=\int_0^1\left|f'(t)-g'(t)\right|\ dt\]
for all $f,g\in\AC_+(I)$. In the proof of \cite[Lemma 2.4]{Solecki} Solecki shows $d_*$ defines a right-invariant metric on $\AC_+(I)$ which induces a Polish group topology.

\begin{lemma}\label{lemma: inclusion is continuous}
On $\AC_+(I)$ the metrics $d_\infty$ and $d_*$ satisfy $d_\infty\leq d_*$.  Consequently, inclusion $\AC_+(I)\to\Homeo_+(I)$ is continuous.
\end{lemma}

\begin{proof}
Let $f,g\in\AC_+(I)$ and let $x\in I$.  Then
\begin{align*}
f(x)-g(x)=\int_0^x f'(t)-g'(t)\ dt\leq\int_0^x|f'(t)-g'(t)|\ dt\leq d_*(f,g)
\end{align*}
and likewise $g(x)-f(x)\leq d_*(f,g)$.  So $d_\infty\leq d_*$ on $\AC_+(I)$.  This says inclusion $(\AC_+(I),d_*)\to(\Homeo_+(I),d_\infty)$ is a contraction mapping and so it is continuous.
\end{proof}

We set
\[d_*(f,g)=\int_0^1\left|f'(t)-g'(t)\right|\ dt\]
for all $f,g\in\AC_\Z^\loc(I)$.  In other words, on $\AC_\Z^\loc(\R)$ the quantity $d_*$ is given by the same formula which defines the metric $d_*$ on $\AC_+(I)$.  It is straightforward to check that $d_*$ defines a pseudometric on $\AC_\Z^\loc(\R)$.  For $f,g\in\AC_\Z^\loc(\R)$ we have $d_*(f,g)=0$ if and only if $f-g$ is a constant function, so $d_*$ does not define a metric on $\AC_\Z^\loc(\R)$.

For each $r\in\R$ we let $\tau_r:\R\to\R$ be translation $x\mapsto x+r$.

\begin{prop}\label{prop: d_*}
On $\AC_\Z^\loc(\R)$ the pseudometric $d_*$ is right-invariant and satisfies
\[d_*(\tau_r\circ f,\tau_s\circ g)=d_*(f,g)\]
for all $r,s\in\R$ and all $f,g\in\AC_\Z^\loc(\R)$.
\end{prop}

\begin{proof}
For each $f\in\AC_\Z^\loc(\R)$ we have
\[f'=(\tau_1\circ f)'=(f\circ\tau_1)'=f'\circ\tau_1\]
so $f'$ is periodic with period $1$.  From this it follows that the integrand $|f'(t)-g'(t)|$ that appears in the definition of $d_*$ is periodic in $t$ with period $1$ and so
\[d_*(f,g)=\int_a^b|f'(t)-g'(t)|\ dt\]
for all $f,g\in\AC_\Z^\loc(\R)$ and all $a,b\in\R$ with $b-a=1$.  Now for any $f,g,u\in\AC_\Z^\loc(\R)$
\begin{align*}
d_*(f\circ u,g\circ u)&=\int_0^1\left|(f\circ u)'(t)-(g\circ u)'(t)\right|\ dt\\
&=\int_{u(0)}^{u(1)}\left|f'(t)-g'(t)\right| dt\\
&=d_*(f,g)
\end{align*}
by integration by substitution and because $u(1)-u(0)=1$, and so $d_*$ is right-invariant.  For all $r\in\R$ and $f\in\AC_\Z^\loc(\R)$ we have $(\tau_r\circ f)'=f'$ which implies the equality in the proposition.
\end{proof}

We seek a right-invariant metric on $\AC_\Z^\loc(\R)$ which induces a Polish group topology and there is an obvious candidate for such a metric.  The sum of a right-invariant metric and right-invariant pseudometric is always a right-invariant metric, so $d_\infty+d_*$ defines a right-invariant metric on $\AC_\Z^\loc(\R)$.  We also set
\[d_*(f,g)=\int_0^1\left|\tilde f'(t)-\tilde g'(t)\right|\ dt\]
for all $f,g\in\AC_+\left(\mathbb S^1\right)$.  Following similar reasoning, $d_*$ defines a right-invariant pseudometric and $d_\infty+d_*$ defines a right-invariant metric on $\AC_+\left(\mathbb S^1\right)$.

We prove for both $G=\AC_+\left(\mathbb S^1\right)$ and $G=\AC_\Z^\loc\left(\R\right)$ the metric $d_\infty+d_*$ is compatible with a Polish group topology on $G$.  See Propositions \ref{prop: line group topology} and \ref{prop: circle group topology}.  We note here that Lemma \ref{lemma: inclusion is continuous} implies the metrics $d_*$ and $d_\infty+d_*$ induce the same topology on $\AC_+(I)$.

\section{$\AC_+(I)$ is a Coarsely Bounded Group}\label{sect: absolute continuity}

To prove that $\AC_+(I)$ is a coarsely bounded group we find a suitably uniform deformation retract of $\AC_+(I)$ onto its trivial subgroup.  Lemma \ref{lem:  homotopy implies coarsely bounded} serves to isolate the uniformity condition.

\begin{lemma}\label{lem:  homotopy implies coarsely bounded}
Suppose $G$ is a Polish group and $\mathcal F:G\times I\to G$ is a deformation retract of $G$ onto its trivial subgroup.  Also suppose for every identity neighborhood $U\subset G$ there exists $\epsilon>0$ such that for all $g\in G$ and all $r,s\in I$, if $|r-s|\leq \epsilon$ then $\mathcal F(g,r)\mathcal F(g,s)^{-1}\in U$.  Then $G$ is a coarsely bounded group.
\end{lemma}

\begin{proof}
Let $U\subset G$ be an identity neighborhood, let $\epsilon>0$ be given by the condition in the lemma, and let $n>0$ be a positive integer with $n\geq \epsilon^{-1}$.  For every $g\in G$ we have
\begin{align*}
g=\mathcal F(g,0)\mathcal F(g,1)^{-1}=\prod_{i=1}^n\left[\mathcal F\left(g,\frac{i-1}{n}\right)\mathcal F\left(g,\frac{i}{n}\right)^{-1}\right]\in U^n
\end{align*}
so $G\subset U^n$, and so the condition implies the following:  For every identity neighborhood $U$ there is an integer $n>0$ such that $G= U^n$.

Let $d$ be a compatible, left-invariant metric on $G$.  We must show $G$ is bounded in $d$.  So let $B$ denote the open unit $d$--ball centered at the identity in $G$.  By the condition at the end of the previous paragraph there is an integer $n>0$ such that $G=B^n$.  By left invariance and the triangle inequality this implies the $d$--diameter of $G$ is no greater than $n$.
\end{proof}

\begin{prop}\label{prop: AC is coarsely bounded}
$\AC_+(I)$ is a coarsely bounded group.
\end{prop}

\begin{proof}
For every $f\in\AC_+(I)$ and $r\in I$ let $\mathcal F(f,r)$ be the element of $\AC_+(I)$ defined by
\[\mathcal F(f,r)(x)=(1-r)f(x)+rx\]
for all $x\in I$.  Every $f\in\AC_+(I)$ is increasing so $f'\geq 0$.  It follows
\begin{align*}
d_*\left(\mathcal F(f,r),\mathcal F(f,s)\right)&=|r-s|\ \int_0^1\left|f'(t)-1\right| dt\\
&\leq|r-s|\ \left(\int_0^1|f'(t)|\ dt+\int_0^1 1\ dt\right)\\
&=|r-s|\left(\int_0^1f'(t)\ dt+1\right)\\
&=|r-s|\ 2
\end{align*}
for all $f\in \AC_+(I)$ and $r,s\in I$.

Now apply Lemma \ref{lem:  homotopy implies coarsely bounded} with $G=\AC_+(I)$.  It is true that $\mathcal F$ is a deformation retract of $G$ onto its trivial subgroup but we omit verifying this.  Note that the proof of Lemma \ref{lem:  homotopy implies coarsely bounded} does not involve continuity of $\mathcal F$ and so it is enough to check that the uniformity condition on $\mathcal F$ is satisfied.  Let $U$ be any identity neighborhood in $G$, let $B$ be the open ball of radius $\delta$ about the identity in $G$ with $\delta$ chosen sufficiently small so that $B\subset U$, and take $\epsilon=\delta/2$.  From the above inequality it follows that for any $r,s\in I$, if $|r-s|\leq \epsilon$ then $d_*(\mathcal F(f,r),\mathcal F(f,s))\leq \delta$, and so by right invariance of $d_*$ we have $\mathcal F(f,r)\mathcal F(f,s)^{-1}\in U$.  So Lemma \ref{lem:  homotopy implies coarsely bounded} applies and $\AC_+(I)$ is a coarsely bounded group.
\end{proof}

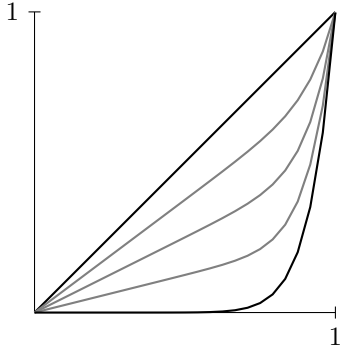
\begin{figure}[ht]
\centering
\begin{tikzpicture}[xscale=4.0,yscale=4.0]
\draw  (0,1.0) -- (0,0) -- (1.0,0);
\draw (-0.02,1.0) --(0.02,1.0);
\node [left] at (-0.02,1) {1};
\draw (1.0,0.02) --(1.0,-0.02);
\node [below] at (1,-0.02) {1};

\draw[black, thick, domain=0:1] plot (\x, {\x});
\draw[gray, thick, domain=0:1] plot (\x, {(1/4)*\x^12+(3/4)*\x});
\draw[gray, thick, domain=0:1] plot (\x, {(1/2)*\x^12+(1/2)*\x});
\draw[gray, thick, domain=0:1] plot (\x, {(3/4)*\x^12+(1/4)*\x});
\draw[black, thick, domain=0:1] plot (\x, {\x^12});
\end{tikzpicture}
\caption{The deformation retract $\mathcal F$ of $\AC_+(I)$ onto its trivial subgroup that appears in the proof of Proposition \ref{prop: AC is coarsely bounded} deforms the graph of a homeomorphism towards the diagonal.  The graphs of $f:I\to I,x\mapsto x^{12}$ and the identity $I\to I$ are shown in black and the graphs of $\mathcal F(f,r)$ for $r=1/4,1/2,$ and $3/4$ are shown in gray.}
\end{figure}
A Polish group $G$ is {\em Roelcke precompact} if for every identity neighborhood $U\subset G$ there is a finite set $F\subset G$ so that $G=UFU$.  As noted in \cite{RoelckeDierolf} the Polish group $\Homeo_+(I)$ is Roelcke precompact.

\begin{question}
Is $\AC_+(I)$ Roelcke precompact?
\end{question}

\section{Topology and Geometry of the Zappa-Sz\'ep Product}\label{sect: ZSP}

A group $G$ is a {\em Zappa-Sz\'ep product} of subgroups $H$ and $K$ if $H\cap K=\{1\}$ and $G=HK$ \cite{Szep, Zappa} or equivalently, if the group operation $G\times G\to G$ restricts to a bijection $H\times K\to G$.  Every semidirect product is a Zappa-Sz\'ep product, and as with the semidirect product there is both an internal and external definition for the Zappa-Sz\'ep product.  Suppose $H$ and $K$ are any groups and $\alpha:K\times H\to H$ and $\beta:K\times H\to K$ are functions.  On $H\times K$ define a binary operation by
\[(h_1,k_1)(h_2,k_2)=(h_1\alpha(k_1,h_2),\beta(k_1,h_2)k_2)\]
for all $h_1,h_2\in H$ and all $k_1,k_2\in K$.  If (and only if) this operation makes $H\times K$ a group and also makes the injections
\[H\to H\times K,h\mapsto (h,1_K)\]
and
\[K\to H\times K,k\mapsto(1_H,k)\]
group homomorphisms, then the {\em external Zappa-Sz\'ep product} of $H$ and $K$ with respect to $\alpha$ and $\beta$ is $H\times K$ equipped with this operation.  The identity element in the external product is $(1_H,1_K)$ and the inverse of $(h,k)$ is
\[(\alpha(k^{-1},h^{-1}),\beta(k^{-1},h^{-1}))\]
for all $h\in H$ and all $k\in K$.  Of course, the external product of $H$ and $K$ is an internal product of the subgroups $H\times\{1_K\}$ and $\{1_H\}\times K$.

Given an internal Zappa-Sz\'ep product $G$ of subgroups $H$ and $K$ there are functions $\alpha:K\times H\to H$ and $\beta:K\times H\to K$ which are uniquely determined by the equation
\[kh=\alpha(k,h)\beta(k,h)\]
for all $h\in H$ and $k\in K$.  The binary operation defined above on $H\times K$ makes it the external Zappa-Sz\'ep product with respect to $\alpha$ and $\beta$ and makes the bijection $H\times K\to G,(h,k)\mapsto hk$ a group isomorphism.  It follows that there is a natural correspondence between internal and external Zappa-Sz\'ep products which is analogous to the correspondence which holds for semidirect products.

For a group $G$ with subgroups $H$ and $K$ the subset $HK$ of $G$ is a subgroup if and only if $HK=KH$, so in the definition of the Zappa-Sz\'ep product the two factor subgroups play a symmetric role.  We set aside some notation that will help us study our main examples in which the two factor subgroups play very different roles.

For any set $X$ we let $\Bij(X)$ denote the group of all bijections of $X$ with composition as the group operation.  For any $S,T\subset \Bij(X)$ and $Y\subset X$ we let
\[S\circ T=\{f\circ g\vert f\in S, g\in T\}\]
and
\[S(Y)=\{f(x)\vert f\in S,x\in Y\}.\]
So if $X=H$ is a group then $\Bij(H)$ denotes the group of all bijections of the underlying set $H$.  Similarly if $H$ is a topological group then $\Homeo(H)$ denotes the group of all homeomorphisms of the underlying space $H$.  For a group $H$ and an element $h\in H$ we let $\lambda_h:H\to H$ be left translation $h\mapsto hx$, and we let
\[\Lambda_H=\left\{\lambda_h\vert h\in H\right\}\]
denote the subgroup of $\Bij(H)$ consisting of left translations.

\begin{notation}\label{notation}
Suppose $H$ is a group and $G\leq\Bij(H)$ is a subgroup such that $\Lambda_H\leq G$.  We let
\[K_G=\left\{k\in G\vert k\left(1_H\right)=1_H\right\}\]
denote the isotropy subgroup of $1_H$ in $G$.  We let
\[\Omega:H\times K_G\to G\]
be the function $(h,k)\mapsto \lambda_h\circ k$.
\end{notation}

The assumption $\Lambda_H\leq G$ in Notation \ref{notation} ensures that $G$ is a Zappa-Sz\'ep product of $\Lambda_H$ and $K_G$.  It is clear that $\Lambda_H\cap K_G$ is the trivial subgroup of $G$.  To see $G=\Lambda_H\circ K_G$ note that any $g\in G$ may be decomposed
\[g=\lambda_{g\left(1_H\right)}\circ\left(\lambda_{g\left(1_H\right)}^{-1}\circ g\right)\]
and the composition in parentheses is an element of $K_G$.  Similarly the assumption $\Lambda_H\leq G$ ensures that $\Omega$ is a bijection.

In Section \ref{sect: AC(S) and AC_Z(R)} we take $G=\AC_+\left(\mathbb S^1\right)$ (with $H=\mathbb S^1$) and $G=\AC_\Z^\loc(\R)$ (with $H=\R$).  In both of these cases the isotropy subgroup $K_G$ is isomorphic to $\AC_+(I)$, and as there is already a known Polish group topology on $\AC_+(I)$ the bijection $\Omega$ suggests one should consider the product topology $H\times\AC_+(I)$ on $G$.  Before moving on to these examples we take a few moments to record a number of observations about the topology and geometry of the Zappa-Sz\'ep product.

Suppose $G$ is a topological group which is a Zappa-Sz\'ep product of subgroups $H$ and $K$, then the group operation restricts to a continuous bijection $H\times K\to G$ on the product space $H\times K$.  The following is \cite[Theorem A.3]{RoseCoarse}.

\begin{thm}[Rosendal]\label{thm: group op is homeo}
Suppose $G$ is a Polish group which is a Zappa-Sz\'ep product of closed subgroups $H$ and $K$.  Then the group operation is a homeomorphism $H\times K\to G$.
\end{thm}

\begin{coro}\label{coro: product of closures is closure of product}
Suppose $G$ is a Polish group which is a Zappa-Sz\'ep product of closed subgroups $H$ and $K$.  Then $\overline{S}\thinspace\overline{T}=\overline{ST}$ for any subsets $S\subset H$ and $T\subset K$.
\end{coro}

\begin{proof}
Let $S\subset H$ and $T\subset K$.  Continuity of the group operation $H\times K\to G$ implies $\overline S\thinspace\overline T\subset\overline{ST}$, and by Theorem \ref{thm: group op is homeo} $\overline{S}\thinspace\overline{T}$ is a closed subset of $G$ which contains $ST$, so $\overline{ST}\subset \overline{S}\thinspace\overline{T}$.
\end{proof}

\begin{coro}\label{coro: conditions on Polish group of homeomorphisms}
Suppose $H$ is a Polish group and $G\leq\Homeo(H)$ is a subgroup such that $\Lambda_H\leq G$.  Let $\Omega:H\times K_G\to G$ be the bijection from Notation \ref{notation}.  Suppose $G$ is equipped with a Polish group topology such that $H\to\Lambda_H,h\mapsto \lambda_h$ is a homeomorphism onto the subspace $\Lambda_H$ and $K_G$ is closed.  Then $\Omega$ is a homeomorphism.
\end{coro}

\begin{proof}
Let $\Phi:H\times K_G\to \Lambda_H\times K_G$ be $(h,k)\mapsto(\lambda_h,k)$ and let $\Psi:\Lambda_H\times K_G\to G$ be composition of functions.

By assumption $H\to\Lambda_H,h\mapsto\lambda_h$ is a homeomorphism and so $\Phi$ is a homeomorphism as well.  By virtue of being Polish it follows that $\Lambda_H$ is a closed subgroup of $G$ and so applying Theorem \ref{thm: group op is homeo} we get that $\Psi$ is also a homeomorphism.  As $\Omega=\Psi\circ\Phi$ this implies $\Omega$ is a homeomorphism.
\end{proof}

Without the topological assumptions in Corollary \ref{coro: conditions on Polish group of homeomorphisms} it is possible to have a subgroup $G\leq\Homeo_+(H)$ and a Polish group topology on $G$ for which $\Omega$ is not a homeomorphism.  If $H$ is a group that supports multiple Polish group topologies then considering one topology on $H$ and another topology on $G=\Lambda_H$ yields such a counterexample.

\begin{prop}\label{prop: product topology is group}
Suppose $H$ is a topological group and $G\leq\Homeo(H)$ is a subgroup such that $\Lambda_H\leq G$.  Let $\Omega:H\times K_G\to G$ be the bijection from Notation \ref{notation}.  Also suppose there is a topology on $K_G$ which makes it a topological group.  Then there is a unique topology on $G$ which makes $\Omega$ a homeomorphism.  This topology makes $G$ a topological group if and only if evaluation
\[K_G\times H\to H,(k,h)\mapsto k(h)\]
and the function
\[K_G\times H\to K_G,(k,h)\mapsto\lambda_{k(h)}^{-1}\circ k\circ\lambda_h\]
are continuous.  If $d_H$ and $d_K$ are compatible metrics on $H$ and $K_G$, respectively, then $d$ defined
\[d(f,g)=d_H\left(f\left(1_H\right),g\left(1_H\right)\right)+d_K\left(\lambda_{f\left(1_H\right)}^{-1}\circ f,\lambda_{g\left(1_H\right)}^{-1}\circ g\right)\]
is a compatible metric on $G$.
\end{prop}

\begin{proof}
The unique topology on $G$ which makes $\Omega$ a homeomorphism is clearly the one obtained by declaring $U\subset G$ open if and only if $\Omega^{-1}(U)\subset H\times K_G$ is open.

For all $h\in H$ and all $k\in K_G$ let $\phi(k,h)=k(h)$ and let $\psi(k,h)=\lambda_{k(h)}^{-1}\circ k\circ\lambda_h$.  Let $\otimes$ be the binary operation on $H\times K_G$ defined by
\[(h_1,k_1)\otimes(h_2,k_2)=(h_1\phi(k_1,h_2),\psi(k_1,h_2)\circ k_2)\]
for all $h_1,h_2\in H$ and all $k_1,k_2\in K_G$.  Then
\begin{align*}
\Omega(h_1,k_1)\circ\Omega(h_2,k_2)&=\lambda_{h_1}\circ k_1\circ\lambda_{h_2}\circ k_2\\
&=\lambda_{h_1k_1(h_2)}\circ\left(\lambda_{k_1(h_2)}^{-1}\circ k_1\circ \lambda_{h_2}\circ k_2\right)\\
&=\Omega\left( h_1\phi(k_1,h_2),\psi(k_1,h_2)\circ k_2\right)\\
&=\Omega\left((h_1,k_1)\otimes(h_2,k_2)\right)
\end{align*}
for all $h_1,h_2\in H$ and all $k_1,k_2\in K_G$.  This says $\Omega:(H\times K_G,\otimes)\to (G,\circ)$ is an operation-preserving bijection and thus a group isomorphism.  Indeed, $(H\times K_G,\otimes)$ is the external Zappa-Sz\'ep product of $H$ and $K_G$ with respect to $\phi$ and $\psi$.

The proposition states the equivalence between (1) and (3) among the following three equivalent conditions.

\begin{enumerate}
\item $G$ is a topological group with the product topology from $\Omega$.
\item $(H\times K_G,\otimes)$ is a topological group with the product topology.
\item The functions $\phi$ and $\psi$ are continuous.
\end{enumerate}

The equivalence between (1) and (2) is immediate because $\Omega$ is a group isomorphism and a homeomorphism.

Suppose for a moment that $H$ and $K$ are some arbitrary topological groups and $G=H\times K$ is an external Zappa-Sz\'ep product with respect to functions $\alpha:K\times H\to H$ and $\beta:K\times H\to K$. The claim is that the product topology $H\times K$ makes $G$ a topological group if and only if $\alpha$ and $\beta$ are continuous.  For one direction, if $\alpha$ and $\beta$ are continuous then the group operation and inversion in $G$ have continuous coordinate functions and so are continuous themselves.  For the reverse direction, if the group operation in $G$ is continuous with respect to the product topology $H\times K$ then
\[(k,h)\mapsto (1_H,k)(h,1_K)=(\alpha(k,h),\beta(k,h))\]
defines a continuous function $K\times H\to H\times K$ and so $\alpha$ and $\beta$ are continuous.  Now returning to the setting of the current proposition, $(H\times K_G,\otimes)$ is the external Zappa-Sz\'ep product with respect to $\phi$ and $\psi$ so (2) and (3) are equivalent.

If $d_H$ and $d_K$ are compatible metrics on $H$ and $K_G$, respectively, then $D$ defined
\[D((h_1,k_1),(h_2,k_2))=d_H(h_1,h_2)+d_K(k_1,k_2)\]
for all $h_1,h_2\in H$ and all $k_1,k_2\in K_G$ is a metric on $H\times K$ which is compatible with the product topology.  With $d$ as in the proposition we have
\[D((h_1,k_1),(h_2,k_2))=d(\Omega(h_1,k_1),\Omega(h_2,k_2))\]
so $d$ is a metric on $G$ and $\Omega:(H\times K_G, D)\to (G, d)$ is an isometry.  Hence $d$ is compatible with the topology on $G$.
\end{proof}

The product of two Polish spaces is a Polish space, and so in our applications of Proposition \ref{prop: product topology is group} once we know $G$ is a topological group it becomes obvious it is a Polish group.

We now turn our attention to the large-scale geometry of the Zappa-Sz\'ep product.  Recall that a subset of a group is {\em symmetric} if it is closed under inversion.

\begin{prop}\label{prop: word metric}
Suppose $G$ is a Zappa-Sz\'ep product of subgroups $H$ and $K$.  Also suppose $H$ is generated by a symmetric subset $S\subset H$ and $K$ is generated by a symmetric subset $T\subset K$ with $1\in S\cap T$ and $ST=TS$.  Then the word metric $\rho_{ST}$ is defined on $G$ and
\[\rho_{ST}(hk,1)=\max\{\rho_{S}(h,1),\rho_{T}(k,1)\}\]
for all $h\in H$ and $k\in K$.  Consequently, the inclusions of the two factor subgroups $(H,\rho_S)\to(G,\rho_{ST})$ and $(K,\rho_T)\to(G,\rho_{ST})$ are isometric embeddings.
\end{prop}

\begin{proof}For any integer $n\geq 0$ we have $S^nT^n=(ST)^n$ by repeatedly applying the assumption $ST=TS$.

Let $h\in H$ and $k\in K$ and set $M=\max\{\rho_{S}(h,1),\rho_{T}(k,1)\}$.  Because $1\in S\cap T$ we have $hk\in S^MT^M=(ST)^M$ so $\rho_{ST}$ is defined on $G$ and $\rho_{ST}(hk,1)\leq M$.  If $n\geq0$ is an integer with $\rho_{ST}(hk,1)\leq n$ then $hk\in (ST)^n=S^nT^n$ so there exists $s_1,\ldots,s_n\in S$ and $t_1,\ldots, t_n\in T$ with
\[hk=s_1\cdots s_nt_1\cdots t_n\]
and becuase the group opertion $H\times K\to G$ is injective this implies $h=s_1\cdots s_n$ and $k=t_1\cdots t_n$, so $\rho_S(h,1)\leq n$ and $\rho_T(k,1)\leq n$.  This holds for any $n\geq0$ so $\rho_{ST}(hk,1)\geq \rho_S(h,1)$ and $\rho_{ST}(hk,1)\geq\rho_T(k,1)$, and hence $\rho_{ST}(hk,1)\geq M$.  This proves the equality in the proposition.

Now
\[\rho_{ST}(h_1^{-1}h_2,1)=\max\{\rho_S(h_1^{-1}h_2,1),\rho_T(1,1)\}=\rho_S(h_1^{-1}h_2,1)\]
for all $h_1,h_2\in H$.  By left invariance it follows that inclusion $(H,\rho_S)\to(G,\rho_{ST})$ is an isometric embedding.  Similarly $(K,\rho_T)\to (G,\rho_{ST})$ is an isometric embedding.
\end{proof}

\begin{thm}\label{thm: QI to factor subgroup}
Suppose $G$ is a Polish group which is a Zappa-Sz\'ep product of closed subgroups $H$ and $K$.  Also suppose $H$ is generated by a subset $S\subset H$ which is coarsely bounded in $H$, $K$ is a coarsely bounded group when equipped with the subspace topology, and $SK=KS$.  Then inclusion $H\to G$ is a quasi-isometry of Polish groups.
\end{thm}

\begin{proof}Set $\mathscr S=S\cup\{1\}\cup S^{-1}$.  As $K=K^{-1}$ and $SK=KS$ we have
\[S^{-1}K=(KS)^{-1}=(SK)^{-1}=KS^{-1}\]
and
\[\mathscr S K=SK\cup K\cup S^{-1}K=KS\cup K\cup KS^{-1}=K\mathscr S\]
so by Corollary \ref{coro: product of closures is closure of product}
\[\overline{\mathscr S}K=\overline{\mathscr SK}=\overline{K\mathscr S}=K\overline{\mathscr S}\]
because $\overline{\mathscr S}$ and $K$ are closed.  Now by Proposition \ref{prop: word metric} inclusion $(H,\rho_{\overline{\mathscr S}})\to (G,\rho_{\overline{\mathscr S}K})$ is an isometric embedding.

For every $g\in G$ there is $h\in H$ and $k\in K$ with $g=hk$, and by left invariance
\[\rho_{\overline{\mathscr S}K}(g,h)=\rho_{\overline{\mathscr S}K}(k,1)\leq 1\]
so inclusion $(H,\rho_{\overline{\mathscr S}})\to (G,\rho_{\overline{\mathscr S}K})$ is coarsely onto.  Hence inclusion $(H,\rho_{\overline{\mathscr S}})\to (G,\rho_{\overline{\mathscr S}K})$ is a quasi-isometry of metric spaces.

As $\overline{\mathscr S}$ is a symmetric generating set for $H$ which is closed and coarsely bounded in $H$ the quasi-isometry type of $H$ is that of $(H,\rho_{\overline{\mathscr S}})$.  It remains to show that the quasi-isometry type of $G$ is that of $(G,\rho_{\overline{\mathscr S}K})$.  We know $\overline{\mathscr S}K=\overline{\mathscr S K}$ is a generating set for $G$ which is closed, so we must show $\overline{\mathscr S}K$ is coarsely bounded in $G$.  Let $d$ be a compatible, left-invariant metric on $G$.  Then $d$ restricts to a compatible, left-invariant metric on both $H$ and $K$.  The subsets $\overline{\mathscr S}$ and $K$ are coarsely bounded in $H$ and $K$, respectively, so these subsets are bounded in $d$.  By left invariance and the triangle inequality it follows that $\overline{\mathscr S}K$ is bounded in $d$, and thus $\overline{\mathscr S}K$ is coarsely bounded in $G$.  So the quasi-isometry type of $G$ is that of $(G,\rho_{\overline{\mathscr S}K})$ as required to make inclusion $H\to G$ a quasi-isometry of Polish groups.
\end{proof}

\begin{thm}\label{thm: QI of group of bijections}
Suppose $H$ is a Polish group and $G\leq\Homeo(H)$ is a subgroup such that $\Lambda_H\leq G$.  Let $K_G\leq G$ be the isotropy subgroup from Notation \ref{notation}.  Suppose $G$ is equipped with a Polish group topology which satisfies the assumptions of Corollary \ref{coro: conditions on Polish group of homeomorphisms}.  Also suppose $H$ is generated by a subset $S\subset H$ which is coarsely bounded in $H$, $K_G$ is a coarsely bounded group when equipped with the subspace topology, $K_G(S)\subset S$, and $K_G\left(S^{-1}\right)\subset S^{-1}$.  Then $H\to G,h\mapsto\lambda_h$ is a quasi-isometry of Polish groups.
\end{thm}

\begin{proof}
First note that for all $h\in H$ and $k\in K_G$
\[k\circ \lambda_h=\lambda_{k(h)}\circ\left(\lambda_{k(h)}^{-1}\circ k\circ \lambda_h\right)\]
and
\[\lambda_h\circ k=\left(\lambda_h\circ k\circ\lambda_{k^{-1}(h^{-1})}\right)\circ\lambda_{k^{-1}(h^{-1})}^{-1}\]
and in both equations the composition in parentheses is an element of $K_G$.

Set $\Lambda_S=\{\lambda_s\vert s\in S\}$.  By assumption $k(s)\in S$ and $k^{-1}(s^{-1})\in S^{-1}$ for every $s\in S$ and $k\in K_G$, so the first equation above implies $K_G\circ\Lambda_S\subset\Lambda_S\circ K_G$ and the second equation implies $\Lambda_S\circ K_G\subset K_G\circ\Lambda_S$, and thus $\Lambda_S\circ K_G=K_G\circ \Lambda_S$.

By Theorem \ref{thm: QI to factor subgroup} with $\Lambda_S$ in place of $S$ it follows that inclusion $\Lambda_H\to G$ is a quasi-isometry of Polish groups.  As $H\to\Lambda_H,h\mapsto\lambda_h$ is an isomorphism of Polish groups it is a quasi-isometry of Polish groups, and so Theorem \ref{thm: QI of group of bijections} follows by composing the quasi-isometry $H\to\Lambda_H$ with the quasi-isometry $\Lambda_H\to G$.
\end{proof}

We make note of a few examples of Theorems \ref{thm: QI to factor subgroup} and \ref{thm: QI of group of bijections}.  First, suppose $G$ is a Polish group which is a semidirect product of closed subgroups $H$ and $K$ with $K$ normal, so $SK=KS$ for every subset $S\subset G$.  If $H$ is generated by a subset which is coarsely bounded in $H$ and $K$ is a coarsely bounded group when equipped with the subspace topology, then Theorem \ref{thm: QI to factor subgroup} says inclusion $H\to G$ is a quasi-isometry of Polish groups.  The Zappa-Sz\'ep products of Example \ref{ex: Isom(H,d)} need not be semidirect products.

\begin{example}\label{ex: Isom(H,d)}
Let $H$ be a Polish group which admits a complete, compatible, left-invariant metric $d$ whose closed unit ball generates $H$, and let $G$ be the Polish group obtained by equipping the isometry group $\Isom(H,d)$ with the topology of pointwise convergence.  In this case $H\to\Lambda_H,h\mapsto\lambda_h$ is a homeomorphism onto the subspace $\Lambda_H$ and $K_G$ is closed, so the conditions of Corollary \ref{coro: conditions on Polish group of homeomorphisms} are satisfied.  Take $S$ to be the closed unit ball in $(H,d)$.  If $S$ is coarsely bounded in $H$ and $K_G$ is a coarsely bounded group with the subspace topology, then Theorem \ref{thm: QI of group of bijections} says $H\to G,h\mapsto\lambda_h$ is a quasi-isometry of Polish groups.  The condition relating $S$ and $K_G$ in the theorem is satisfied trivially:  For every $h\in H$ and $k\in K_G$
\[d(k(h),1_H)=d(k(h),k\left(1_H\right))=d(h,1_H)\]
which implies $K_G(S)\subset S$, and since $S=S^{-1}$ this also says $K_G(S^{-1})\subset S^{-1}$.
\end{example}

If $H$ is locally compact then the isotropy subgroup of $\Isom(H,d)$ is compact and thus is a coarsely bounded group.  So for instance Example \ref{ex: Isom(H,d)} describes the situation when $H$ is a finitely generated group and $d=\rho_S$ is the word metric with respect to a finite generating set $S$.

For $f\in\Homeo_\Z(\R)$ we let $f\vert_I$ denote the restriction of $f$ to a homeomorphism $I\to f(I)$ and for $f\in\Homeo_+(I)$ we let $\hat f$ denote the homeomorphism $\R\to\R$ which is defined by
\[\hat f(x)=f(x-n)+n\]
for all $n\in\Z$ and all $x\in[n,n+1]$.  It is straightforward to check that $k\mapsto k\vert_I$ defines an isomorphism of topological groups from the isotropy subgroup of $0$ in $\Homeo_\Z(\R)$ to $\Homeo_+(I)$ with inverse isomorphism given by $k\mapsto \hat k$.  See Figure \ref{fig: homeos fixing 0}.

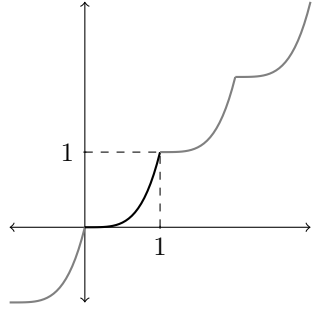
\begin{figure}[ht]
\centering
\begin{tikzpicture}[xscale=1.0,yscale=1.0]
\draw [<->] (-1.0,0.0) -- (0,0) -- (0.0,-1.0);
\draw [<->] (0,3.0) -- (0,0) -- (3.0,0);
\draw (-0.02,1.0) --(0.02,1.0);
\node [left] at (-0.02,1) {1};
\draw (1.0,0.02) --(1.0,-0.02);
\node [below] at (1,-0.02) {1};
\draw [black, dashed] (0.0,1.0) --(1.0,1.0);
\draw [black, dashed] (1.0,0.0) --(1.0,1.0);
\draw[gray, thick, domain=-1:0] plot (\x, {(\x-floor(\x))^4+floor(\x)});
\draw[black, thick, domain=0:1] plot (\x, {\x^4});
\draw[gray, thick, domain=1:2] plot (\x, {(\x-floor(\x))^4+floor(\x)});
\draw[gray, thick, domain=2:3] plot (\x, {(\x-floor(\x))^4+floor(\x)});
\end{tikzpicture}
\caption{The isotropy subgroup of $0$ in $\Homeo_\Z(\R)$ is isomorphic to $\Homeo_+(I)$.  The graph of $f:I\to I,x\mapsto x^4$ is shown in black and the graph of $\hat f:\R\to\R$ is shown in gray.}
\label{fig: homeos fixing 0}
\end{figure}

\begin{example}\label{ex: Homeo_Z(R)}
We apply Theorem \ref{thm: QI of group of bijections} with $H=\R$, $G=\Homeo_\Z(\R)$, and $S=[-1,1]$.  It is straightforward to verify the topological assumptions of Corollary \ref{coro: conditions on Polish group of homeomorphisms}.  The subset $S$ generates $H$ and because it is compact it is also coarsely bounded.  By \cite[Lemma 8]{MannRose} the Polish group $\Homeo_\partial\left(\mathbb B^n\right)$ of homeomorphisms of the compact ball of dimension $n$ which fix the boundary is a coarsely bounded group for any integer $n>0$.  As noted above $K_G$ and $\Homeo_+(I)$ are isomorphic groups and because $\Homeo_+(I)$ and $\Homeo_\partial\left(\mathbb B^1\right)$ are isomorphic it follows that $K_G$ is a coarsely bounded group as well.  For every $s\in S$ and $k\in K_G$ the Intermediate Value Theorem gives $k(s)\in S$ so $K_G(S)\subset S$, and $S=S^{-1}$ so this also says $K_G(S^{-1})\subset S^{-1}$.  Applying Theorem \ref{thm: QI of group of bijections} shows $\R\to\Homeo_\Z(\R),r\mapsto\tau_r$ is a quasi-isometry of Polish groups.
\end{example}

\section{$\AC_+\left(\mathbb S^1\right)$ and $\AC_\Z^{\loc}(\R)$}\label{sect: AC(S) and AC_Z(R)}

In this final section we study the Zappa-Sz\'ep products $\AC_+\left(\mathbb S^1\right)$ and $\AC_\Z^{\loc}(\R)$.  In particular, we verify the claim from Section \ref{sect: prelims} that for $G=\AC_+\left(\mathbb S^1\right)$ and $G=\AC_\Z^{\loc}(\R)$ the right-invariant metric $d_\infty+d_*$ on $G$ is compatible with a Polish group topology.  Then we use Theorem \ref{thm: QI of group of bijections} to describe the quasi-isometry types of these Polish groups.

We set
\[K_*=\left\{k\in\AC_\Z^\loc(\R)\vert k(0)=0\right\}\]
and note that $k\mapsto k\vert_I$ defines an isomorphism of groups $K_*\to\AC_+(I)$ which preserves $d_*$.  As $d_*$ is a right-invariant metric on $\AC_+(I)$ which induces a Polish group topology it follows that $d_*$ is also a right-invariant metric on $K_*$ which induces a Polish group topology, and the Polish groups $K_*$ and $\AC_+(I)$ are isomorphic.

\begin{lemma}\label{lem: eval is continuous}
Evaluation
\[K_*\times\R\to\R,(k,r)\mapsto k(r)\]
is continuous.
\end{lemma}

\begin{proof}
Let
\[K_\infty=\left\{k\in\Homeo_\Z(\R)\vert k(0)=0\right\},\]
let
\[\Phi:K_*\times H\to K_\infty\times H\]
be inclusion, and let
\[\Psi:K_\infty\times H\to H\]
be evaluation.  By Lemma \ref{lemma: inclusion is continuous} inclusion $\AC_+(I)\to\Homeo_+(I)$ is continuous so also $\Phi$ is continuous, and by Proposition \ref{prop: product topology is group} (with $H=\R$ and $G=\Homeo_\Z(\R)$) $\Psi$ is continuous.  Evaluation $K_*\times H\to H$ is the composition $\Psi\circ\Phi$, and so it is continuous.
\end{proof}

\begin{lemma}\label{lem: continuous on product}
The function
\[K_*\times \R\to K_*,(k,r)\mapsto\tau_{k(r)}^{-1}\circ k\circ\tau_r\]
is continuous.
\end{lemma}

\begin{proof}
For all $r\in\R$ and $k\in K_*$ let $\psi(k,r)=\tau_{k(r)}^{-1}\circ k\circ\tau_r$.  Now fix $(k,r)\in K_*\times\R$ and let $\epsilon>0$ be given.  For any compact interval $J$ the collection of continuous functions $\mathcal C(J)$ is a dense subset of $L^1(J)$, so there exists some continuous function $\gamma:[-\epsilon,1+\epsilon]\to \R$ with
\[\int_{-\epsilon}^{1+\epsilon}|k'(t)-\gamma(t)|\ dt<\frac{\epsilon}{4}\]
and by uniform continuity of $\gamma$ there exists some $\delta>0$ so that
\[|\gamma(x)-\gamma(y)|<\frac{\epsilon}{4}\]
for all $x,y\in \R$ with $|x-y|<\delta$.  Using the properties of $d_*$ from Proposition \ref{prop: d_*}, for all $s\in\R$ and $l\in K_*$
\begin{align*}
d_*\left(\psi(k,r),\psi(l,s)\right)=d_*(k,l\circ\tau_{s-r})
\end{align*}
and if $d_*(k,l)+|r-s|<\min\{\delta,\epsilon/4\}$ then
\begin{align*}
d_*(k,l\circ\tau_{s-r})&\leq \int_0^1|k'(t)-\gamma(t)|\ dt\\
&\hspace{.5cm}+\int_0^1|\gamma(t)-\gamma(t+s-r)|\ dt\\
&\hspace{.5cm}+\int_0^1|\gamma(t+s-r)-k'(t+s-r)|\ dt\\
&\hspace{.5cm}+\int_0^1|k'(t+s-r)-l'(t+s-r)|\ dt\\
&<\epsilon
\end{align*}
so $\psi$ is continuous at $(k,r)$.  Because the argument given works for an arbitrary point $(k,r)\in K_*\times\R$ it follows that $\psi$ is continuous on $K_*\times\R$.
\end{proof}

Now we extend the topology on $K_*$ to a Polish group topology on $\AC_\Z^\loc(\R)$.

\begin{prop}\label{prop: line group topology}
Set $H=\R$ and $G=\AC_\Z^\loc(\R)$ and let $\Omega:H\times K_G\to G$ be the bijection from Notation \ref{notation}.  Then the unique topology on $G$ which makes $\Omega$ a homeomorphism also makes $G$ a Polish group, and the metric $d_\infty+d_*$ is compatible with this topology.
\end{prop}

\begin{proof}
In the present notation $K_*=K_G$.  Evaluation $K_G\times H\to H$ and the function
\[K_G\times H\to K_G,(k,h)\mapsto\lambda_{k(h)}^{-1}\circ k\circ\lambda_h\]
are continuous by Lemmas \ref{lem: eval is continuous} and \ref{lem: continuous on product}.  Now Proposition \ref{prop: product topology is group} applies and so the topology that makes $\Omega$ a homeomorphism also makes $G$ a Polish group.  A compatible metric $d$ on $G$ is given by
\[d(f,g)=|f(0)-g(0)|+d_*\left(\tau_{f(0)}^{-1}\circ f,\tau_{g(0)}^{-1}\circ g\right)=|f(0)-g(0)|+d_*(f,g)\]
for all $f,g\in G$.

For all $f,g\in G$
\[d(f,g)=|f(0)-g(0)|+d_*(f,g)\leq d_\infty(f,g)+d_*(f,g)\]
and
\begin{align*}
d_\infty(f,g)\leq&|f(0)-g(0)|+d_\infty\left(\tau_{f(0)}^{-1}\circ f,\tau_{g(0)}^{-1}\circ g\right)\\
\leq&|f(0)-g(0)|+d_*\left(\tau_{f(0)}^{-1}\circ f,\tau_{g(0)}^{-1}\circ g\right)\\
=&|f(0)-g(0)|+d_*(f,g)
\end{align*}
so $d_\infty(f,g)+d_*(f,g)\leq 2\ d(f,g)$.  This implies $d$ and $d_\infty+d_*$ induce the same topology on $G$.
\end{proof}

As an aside we note that the topology induced by the pseudometric $d_*$ on $\AC_\Z^\loc(\R)$ is not a group topology.  To see this, let
\[\mathcal T=\{\tau_r\vert r\in\R\}\]
be the subgroup of $\AC_\Z^\loc(\R)$ consisting of real translations and for $k\in K_*$ consider the cosets $\mathcal T\circ k$ and $k\circ \mathcal T$.  By Proposition \ref{prop: d_*} the right coset $\mathcal T\circ k$ has $d_*$--diameter $0$.  On the other hand, if $k\circ \mathcal T$ has $d_*$--diameter $0$ then the Fundamental Theorem implies the homeomorphism $k$ is also a homomorphism of $(\R,+)$, and so $k$ must be the identity $\R\to\R$.  This says $k\circ\mathcal T$ has $d_*$--diameter $0$ if and only if $k$ is the identity.  From this it follows that inversion in $\AC_\Z^\loc(\R)$ exchanges subsets with $d_*$--diameter $0$ and subsets with positive $d_*$--diameter, and so inversion is not a homeomorphism with the topology induced by $d_*$.

We set
\[K_\circ=\left\{k\in\AC_+\left(\mathbb S^1\right)\vert k(1)=1\right\}\]
and note that $k\mapsto \tilde k\vert_I$ defines an isomorphism of groups $K_\circ\to \AC_+(I)$ which preserves $d_*$.  It follows that $d_*$ is a right-invariant metric on $K_\circ$ which induces a Polish group topology, and the Polish groups $K_\circ$ and $\AC_+(I)$ are isomorphic.  We extend the topology on $K_\circ$ to a Polish group topology on $\AC_+\left(\mathbb S^1\right)$.  Alternatively, one may define the same Polish group topology on $\AC_+\left(\mathbb S^1\right)$ by identifying this group with the quotient of $\AC_\Z^\loc(\R)$ by the closed normal subgroup consisting of integer translations.

\begin{prop}\label{prop: circle group topology}
Set $H=\mathbb S^1$ and $G=\AC_+\left(\mathbb S^1\right)$ and let $\Omega:H\times K_G\to G$ be the bijection from Notation \ref{notation}.  Then the unique topology on $G$ which makes $\Omega$ a homeomorphism also makes $G$ a Polish group, and the metric $d_\infty+d_*$ is compatible with this topology.
\end{prop}

\begin{proof}
In the present notation $K_\circ=K_G$.  Let
\[K_\infty=\left\{k\in\Homeo_+\left(\mathbb S^1\right)\vert k(1)=1\right\}\]
let
\[\Phi:K_G\times H\to K_\infty\times H\]
be inclusion, and let
\[\Psi:K_\infty\times H\to H\]
be evaluation.  Lemma \ref{lemma: inclusion is continuous} implies inclusion $K_G\to K_\infty$ is continuous so also $\Phi$ is continuous, and by Proposition \ref{prop: product topology is group} (with $H=\mathbb S$ and $G=\Homeo_+\left(\mathbb S^1\right)$) $\Psi$ is continuous.  Evaluation $K_G\times H\to H$ is the composition $\Psi\circ\Phi$, and so it is continuous.  The function in Lemma \ref{lem: continuous on product} is continuous and descends to a continuous function $K_*\times\mathbb S^1\to K_*$, so
\[K_G\times H\to K_G,(k,r)\mapsto\lambda_{k(r)}^{-1}\circ k\circ\lambda_r\]
is continuous.  By Proposition \ref{prop: product topology is group} the topology that makes $\Omega$ a homeomorphism also makes $G$ a Polish group.  A compatible metric $d$ on $G$ is given by
\[d(f,g)=d_{\mathbb S^1}(f(1),g(1))+d_*\left(f,g\right)\]
for all $f,g\in G$.

The argument that $d$ and $d_\infty+d_*$ induce the same topology on $G$ works similarly as in the proof of Proposition \ref{prop: line group topology}.
\end{proof}

\begin{prop}\label{prop: AC_+(S) is coarsely bounded}
$\AC_+\left(\mathbb S^1\right)$ is a coarsely bounded group.
\end{prop}

\begin{proof}
Apply Theorem \ref{thm: QI of group of bijections} with $H=S=\mathbb S^1$ and $G=\AC_+\left(\mathbb S^1\right)$.  By Proposition \ref{prop: circle group topology} the assumptions of Corollary \ref{coro: conditions on Polish group of homeomorphisms} are satisfied.  As $S$ is compact it is coarsely bounded.  The isotropy subgroup $K_G$ is isomorphic to $\AC_+(I)$ and so is a coarsely bounded group.  The condition $K_G(S)\subset S$ and $K_G(S^{-1})\subset S^{-1}$ is satisfied trivially because $S=H$.  By Theorem \ref{thm: QI of group of bijections} $H\to G, h\mapsto\lambda_h$ is a quasi-isometry of Polish groups.  Since $H$ is a coarsely bounded group so is $G$.
\end{proof}

\begin{prop}\label{prop: AC_+^loc(Z) qi type}
$\AC_\Z^\loc(\R)$ is quasi-isometric to $\Z$.
\end{prop}

\begin{proof}
Apply Theorem \ref{thm: QI of group of bijections} with $H=\mathbb \R$, $S=[-1,1]$, and $G=\AC_\Z^\loc(\R)$.  By Proposition \ref{prop: line group topology} the assumptions of Corollary \ref{coro: conditions on Polish group of homeomorphisms} are satisfied.  As $S$ is compact it is a coarsely bounded subset of $H$.  The isotropy subgroup $K_G$ is isomorphic to $\AC_+(I)$ and so is a coarsely bounded group.  The condition $K_G(S)\subset S$ and $K_G(S^{-1})\subset S^{-1}$ is satisfied by the Intermediate Value Theorem just the same as in Example \ref{ex: Homeo_Z(R)}.  By Theorem \ref{thm: QI of group of bijections} $H\to G, h\mapsto\lambda_h$ is a quasi-isometry of Polish groups.  Now $\Z\to\AC_\Z^\loc(\R),n\mapsto\tau_n$ is a quasi-isometry of Polish groups because it is the composition of inclusion $\Z\to H$ and $H\to G,h\mapsto\lambda_h$.
\end{proof}

Theorem \ref{thm: Main Theorem} then just collects the statements of Propositions \ref{prop: AC is coarsely bounded}, \ref{prop: AC_+(S) is coarsely bounded}, and \ref{prop: AC_+^loc(Z) qi type}.  In the general theory of \cite{RoseCoarse}, if $G$ is a Polish group which is generated by a coarsely bounded subset then there is always a metric on $G$ which is simultaneously compatible with the topology, right-invariant, and realizes the quasi-isometry type of $G$.  In closing we complete the proof that $d_\infty+d_*$ is such a metric on $\AC_\Z^\loc(\R)$.

\begin{prop}
The metric $d_\infty+d_*$ on $\AC_\Z^\loc(\R)$ is a representative of the quasi-isometry type of $\AC_\Z^\loc(\R)$.
\end{prop}

\begin{proof}
For all $r,s\in\R$
\[d_\infty(\tau_r,\tau_s)+d_*(\tau_r,\tau_s)=|r-s|\]
so $r\mapsto\tau_r$ defines an isometric embedding of $\R$ with its standard metric into $\AC_\Z^\loc(\R)$ with the metric $d_\infty+d_*$.  As $\AC_\Z^\loc(\R)=\mathcal T\circ K_*$ and $K_*$ is bounded in $d_\infty+d_*$ it follows that the isometric embedding of $\R$ into $\AC_\Z^\loc\left(\R\right)$ is coarsely onto, and so the metric $d_\infty+d_*$ on $\AC_\Z^\loc(\R)$ represents the quasi-isometry of type of $\AC_\Z^\loc(\R)$. 
\end{proof}

\bibliographystyle{plain}
\bibliography{references}

\begin{thebibliography}{1}

\bibitem{Cohen}
Michael~P. Cohen.
\newblock On the large-scale geometry of diffeomorphism groups of
  $1$-manifolds.

\bibitem{MannRose}
Kathryn Mann and Christian Rosendal.
\newblock Large scale geometry of homeomorphism groups.
\newblock 2016.

\bibitem{RoelckeDierolf}
W.~Roelcke and S.~Dierolf.
\newblock {\em Uniform Structures on Topological Groups and Their Quotients}.
\newblock McGraw-Hill International Book Co., 1981.

\bibitem{RoseCoarse}
Christian Rosendal.
\newblock {\em Coarse Geometry of Topological Groups}.
\newblock 2017.

\bibitem{Rudin}
Walter Rudin.
\newblock {\em Real \& Complex Analysis}.
\newblock MHHE, 1987.

\bibitem{Solecki}
Slawomir Solecki.
\newblock {\em Polish Group Topologies, in Sets and Proofs (London Mathematical
  Society Lecture Note Series 258)}.
\newblock Cambridge University Press, 1999.
\newblock pp 339--364.

\bibitem{Szep}
J.~Sz\'ep.
\newblock On the structure of groups which can be represented as the product of
  two subgroups.
\newblock {\em Acta Sci. Math. Szeged}, 1950.

\bibitem{Zappa}
G.~Zappa.
\newblock Sulla costruzione dei gruppi prodotto di due dati sottogruppi
  permutabili traloro.
\newblock {\em Atti Secondo Congresso Un. Mat. Ital., Bologna; Edizioni
  Cremonense, Rome}, 1942.

\end{thebibliography}

\end{document}